\documentclass[11pt]{amsart}
\usepackage{fullpage}
\usepackage{color}
\usepackage{graphicx,psfrag} 
\usepackage{color} 
\usepackage{tikz}
\usepackage{pgffor}
\usepackage{todonotes}
\usepackage{eqnarray,amsmath}

\usepackage{perpage}     

\usepackage{wrapfig}
\usepackage{enumerate}
\usepackage{subfigure}
\usepackage{caption}

\usepackage[normalem]{ulem}
\usepackage[makeroom]{cancel} 
\usepackage{amssymb} 

\MakePerPage{footnote}   

\newtheorem{theorem}{Theorem}[section]
\newtheorem{lemma}[theorem]{Lemma}
\newtheorem{corollary}[theorem]{Corollary}
\newtheorem{proposition}[theorem]{Proposition}
\newtheorem{conjecture}[theorem]{Conjecture}
\theoremstyle{definition}
\newtheorem{definition}[theorem]{Definition}

\theoremstyle{remark}
\newtheorem{remark}[theorem]{Remark}

\newtheorem{property}[theorem]{Property}
\newtheorem{question}[theorem]{Question}

\newtheorem{ex-prop}[theorem]{Example-Proposition}

\newtheorem{computation}[theorem]{Computation}

\numberwithin{equation}{section}

\definecolor{gray}{rgb}{.5,.5,.5}

\definecolor{black}{rgb}{0,0,0}

\definecolor{blue}{rgb}{0,0,1}

\definecolor{red}{rgb}{1,0,0}

\definecolor{green}{rgb}{0,1,0}

\definecolor{yellow}{rgb}{1,1,.4}

\definecolor{purple}{rgb}{1,0,1}

\definecolor{gold}{rgb}{.5,.5,.2}

\definecolor{darkgreen}{rgb}{0,.5,0}

\definecolor{greenbean}{RGB}{199, 237, 204}

\definecolor{RED}{rgb}{1,0,0}

\DeclareMathOperator{\Ext}{Ext}
\DeclareMathOperator{\Int}{Int}

\begin{document}

\title{Random knots using Chebyshev billiard table diagrams}

\author{Moshe Cohen}
\address{Department of Electrical Engineering, Technion -- Israel Institute of Technology, Haifa 32000, Israel}
\email{mcohen@tx.technion.ac.il}
\thanks{The first author was supported in part by the funding from the European Research Council under the European Union's Seventh Framework Programme, Grant FP7-ICT-318493-STREP}

\author{Sunder Ram Krishnan}
\address{Department of Electrical Engineering, Technion -- Israel Institute of Technology, Haifa 32000, Israel}
\email{eeksunderram@gmail.com}
\thanks{The second author was supported in part by the funding from the European Research Council under Understanding Random Systems via Algebraic Topology (URSAT), ERC Advanced Grant 320422.}

\begin{abstract}
We use the Chebyshev knot diagram model of Koseleff and Pecker in order to introduce a random knot diagram model by assigning the crossings to be positive or negative uniformly at random.  We give a formula for the probability of choosing a knot at random among all knots with bridge index at most 2. Restricted to this class, we define internal and external reduction moves that decrease the number of crossings of the diagram. We make calculations based on our formula, showing the numerics in graphs and providing evidence for our conjecture that the probability of any knot $K$ appearing in this model decays to zero as the number of crossings goes to infinity.
\end{abstract}


\maketitle

\renewcommand{\thesubfigure}{}
\begin{figure}[h]
\centering
\subfigure{
\begin{tikzpicture}[scale=.75]
    \foreach \i in {0,...,7} {
        \draw [very thin,dashed,lightgray] (\i,1) -- (\i,4);
    }
    \foreach \i in {1,...,4} {
        \draw [very thin,dashed,lightgray] (0,\i) -- (7,\i);
    }
\draw[thick](-.25,.75) -- (3,4) -- (6,1) -- (7,2) -- (5,4) -- (2,1) -- (0,3) -- (1,4) -- (4,1) -- (7,4) -- (7.25,4.25);
\draw[color=white](.5,.5) -- (.75,.75);
\end{tikzpicture}
}
\hspace{2em}
\subfigure{
		\includegraphics[width=4cm]{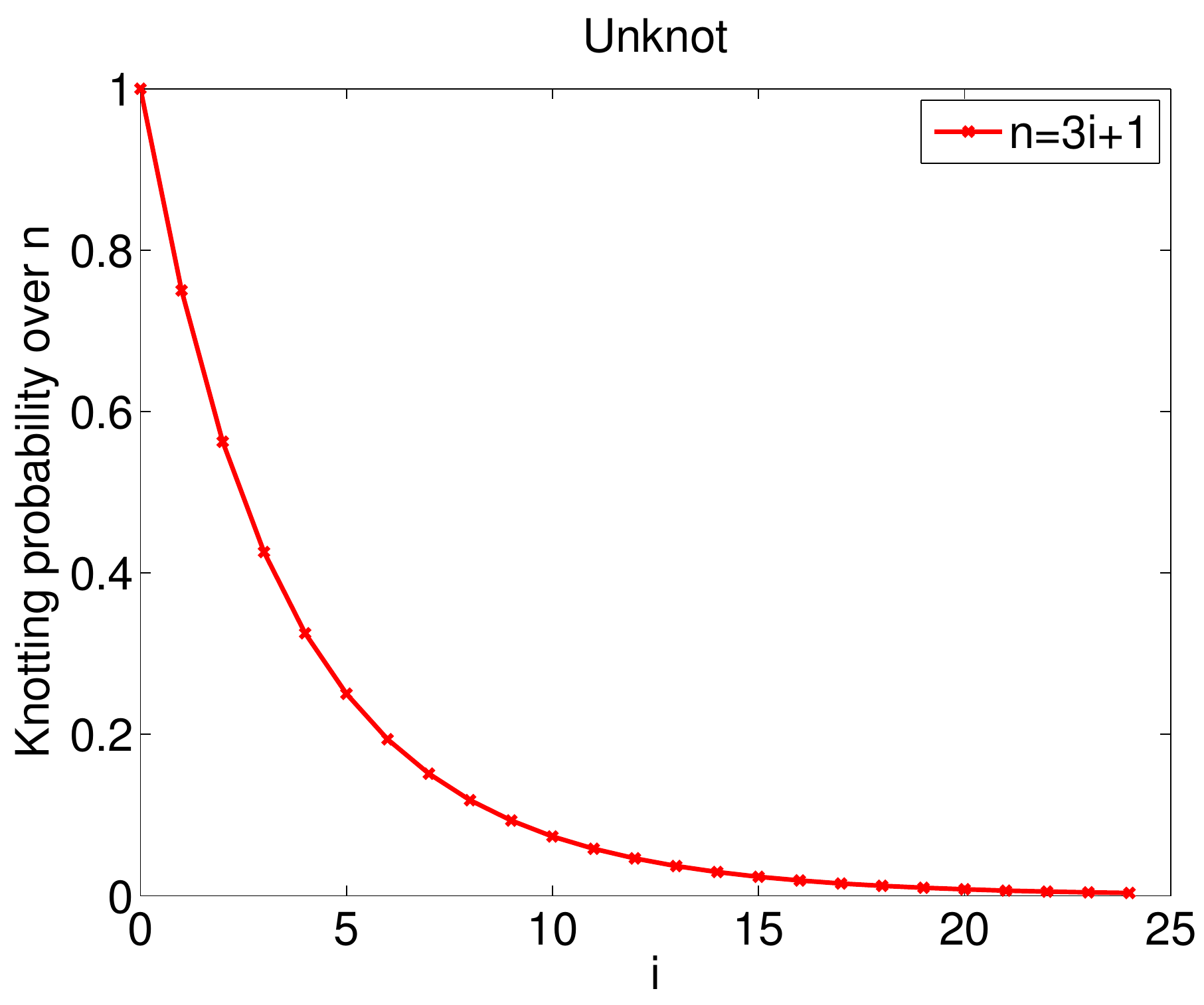}
}
\end{figure}

\section{Introduction}

It is still unknown whether the Jones polynomial detects the unknot \cite[Problem 1]{Jon:unknot}, that is, whether the Jones polynomial of some non-trivial knot is also equal to 1, the Jones polynomial of the unknot.  One might believe this to be true because Khovanov homology, a knot invariant of bigraded abelian groups whose graded Euler characteristic gives the Jones polynomial, has been shown by Kronheimer and Mrowka to detect the unknot \cite{KroMro}.  Dasbach and Hougardy \cite{DasHou} found no counterexample in all knots up to seventeen crossings, including over two million prime knots, and this was extended to eighteen crossings by Yamada \cite{Yam}.  Rolfsen with Anstee and Przytycki \cite{AnPrRo} and with Jones \cite{JonRol} sought a counterexample via mutations.  Bigelow \cite{Big} outlined an algorithm to find a counterexample among four-strand braids.

We propose using the probabilistic method, pioneered by Erd\"os in the 1940's, as a tool to investigate data surrounding this question.  This method has been successful in studying ``typical" large ``random'' graphs, like the Erd\"os-R\'enyi model, and in showing that certain graph theoretic properties hold with certain probabilities for these large graphs.  However, it has been only minimally exploited in topology in general, with some notable exceptions:  the study of random 3-manifolds by Dunfield with W. Thurston \cite{DT:fin} and D. Thurston \cite{DT:tun}; random walks on the mapping class group in a recent work by Ito \cite{Ito}; and the Linial-Meshulam model \cite{LinMes} for random simplicial 2-complexes and its recent generalization by Costa and Farber \cite{CosFar} to $k$-complexes.  Also see survey papers on the topology of random complexes by Kahle \cite{Kah} and Bobrowski and Kahle \cite{BobKah}.

Ultimately, we would like to compare the probability of obtaining the unknot with that of obtaining unit Jones polynomial, but for this we need a useful notion of a space of knots.  Various models for random knots from the literature, appearing more frequently since an AMS Special Session in Vancouver in 1993, generally take one of two forms:  the physical knotting of a random walk, as in \cite{Buck:G, DegTsur:polymer, DSDS:op, DSDS:or, Alias, ABDHKS}, or a random walk on the Cayley graph of the braid group, as in \cite{NecGrosVers, Voit, NecVoit, MairMath}.

Unfortunately, these models do not make common knot theoretic computations accessible.  For this reason, we are interested in constructing a random model based on a particular knot diagram.  The first paper on this topic appeared in 2014 by Even-Zohar, Hass, Linial, and Nowik \cite{EZHLN} using the \"ubercrossing and petal projections of Adams et al. \cite{Adams:petal}.  Forthcoming work of Dunfield, Obeidin, et al. \cite{Dun:knots} is expected to be on the topic of random knot diagrams as well.

In this paper, we apply the probabilistic method to a model developed by Koseleff and Pecker \cite{KosPec, KosPec4, KosPec3} that uses Chebyshev polynomials of the first kind to parametrize a knotted curve.  These so-called \emph{Chebyshev knots} are analogues of the more established Lissajous knots studied by Jones and Przytycki \cite{JonPrz}.  However, Koseleff and Pecker show in \cite{KosPec3} that all knots are Chebyshev, which is not the case for Lissajous knots (see for example \cite{BDHZ}).

The goal of this paper is to establish a framework that could be used in future work to find evidence surrounding the question stated above.  However, the present scope will be restricted to the class of knots whose bridge index is at most 2, and it is already known that the Jones polynomial detects the unknot among this class.   This is because 2-bridge knots are alternating \cite{Goo} and Jones polynomials of alternating knots have spans (or degrees, after some normalization) equal to four times the crossing number of the alternating diagram.  This achieves the minimal crossing number by Tait's Conjecture, proven by Kauffman \cite{KauffmanBracket, Kauff:span}, Murasugi \cite{Murasugi}, and Thistlethwaite \cite{Th}.

One reason in particular why this model is of interest for the question above is that the Jones polynomial follows some nice recursions, as shown with formulas appearing in recent work by the first author \cite{Co:3bridge}.

\medskip

\textbf{Results. }  Let $T(3,n+1)$ be the (crossingless) billiard table diagram (as in Figure \ref{fig:Cheb}) of a Chebyshev knot with bridge index at most 2 for $3\nmid (n+1)$.  Choose uniformly at random an element of $\{+,-\}^n$ to assign positive and negative crossings to the 4-valent vertices of the billiard table, ordered from left to right.

In the result below, we use $x^{(i)}_i=x^{(i)}_i(K)$ that depend on a fixed knot $K$, as in Definition \ref{def:xy}.

\smallskip

\noindent
\textbf{Main Theorem \ref{thm:main} (1).  }  
\textit{The probability of a given knot $K$ appearing in $T(3,n+1)$ is}
\begin{equation*}
\frac{1}{2^n}\left[x_{n+3}^{(n+3)}+\sum_{i=0}^{n-6} \left[\binom{n-5}{i+1}+4\binom{n-5}{i}\right] x_{n-3i}^{(n-3i)} + 4x_{-2n+15}^{(-2n+15)}\right].
\end{equation*}

The second and third parts of this Main Theorem \ref{thm:main} give formulas for the variables $x_i^{(i)}$ when $i$ is positive and nonpositive, respectively.

\medskip

\textbf{Organization. }  Section \ref{sec:knots} starts with a brief background on knots, bridge index, and continued fractions before introducing the relatively new model for Chebyshev knots in Subsection \ref{subsec:Cheb}.  Then in Subsection \ref{subsec:a2}, Definitions \ref{def:int} and \ref{def:ext} introduce the \emph{internal and external reduction moves} that will play a major role in the paper. Theorem \ref{thm:noMoves} then gives the first probability result describing the probability of a knot occurring without any reduction moves.

Definitions \ref{def:xy} and \ref{def:i0} of $S^{(n)}_i$, $x^{(n)}_i$, and $y^{(n)}_i$ in Section \ref{sec:maindefs} are important.  The three tools Lemma \ref {lem:xs}, Equation \ref{eq:fact}, and the $xy$ Lemma \ref{lem:XY} will be used to prove the main theorem.

The primary result of Main Theorem \ref{thm:main} is the first part given by Equation \ref{eq:maineq}.  The other two parts give formulas for the variables $x_i^{(i)}$ that appear in the first part.  Also appearing in Section \ref{sec:mainresults} are Lemma \ref{lem:binom}, which gives an expression for the $x_i^{(n)}$.

Finally Section \ref{sec:data} gives some numerics for Equation \ref{eq:maineq} in Figure \ref{fig:knotprob}, leading to the important Conjecture \ref{conj:0} that the probability of any knot $K$ given by Equation \ref{eq:maineq} approaches zero as $n$ goes to infinity and Question \ref{ques:0rate} concerning the decay rates of the probabilities.

\medskip

\textbf{Acknowledgements. }  Inspiration for this project arose from three places:  many in-depth conversations of the first author with Tahl Nowik and Nati Linial in order to write together a 2011 grant proposal on the subject of random knot diagrams; conversations of the first author with Pierre-Vincent Koseleff (UPMC/INRIA Ouragan), especially those during his visit to Israel in December 2014; and the insistence of Robert Adler on bringing together topologists and probabilists.

\section{The Chebyshev knot model}
\label{sec:knots}

Those already familiar with knots, bridge index, and continued fraction expansions may skip Subsection \ref{subsec:background}.  Chebyshev knots are treated in general in Subsection \ref{subsec:Cheb}, with the $a=3$ case of particular interest for this paper discussed in Subsection \ref{subsec:a2}.

\subsection{Knots, bridge index, and continued fraction expansions}
\label{subsec:background}

A \emph{knot} $K$ is an embedding of the circle in 
 $S^3$, but one can think of the circle embedded in $\mathbb{R}^3$ with the one-point compactification of $\mathbb{R}^3$ away from the circle.  
  A \emph{long knot} is one where this one-point compactification lies on the circle, thus leaving two ends that go to infinity.

  Reidemeister's Theorem allows one to consider knots by focusing solely on their \emph{knot diagrams}, projections to the plane with only 4-valent vertices that are decorated by ``over-'' and ``under-crossing'' information, up to the three Reidemeister moves.

A \emph{bridge} is one of the arcs in a knot diagram: it consists only of over-crossings.  The \emph{bridge index} $br(K)$ of a knot $K$ is the minimum number over all diagrams of disjoint bridges which together include all over-crossings or equivalently the minimum number over all diagrams of local maxima of the knot diagram taken with a Morse function.

All knots with bridge index 2 can be written in \emph{Conway's normal form} $C(a_1,a_2,\ldots,a_{n-1},a_n)$ \cite{Con} given in Figure \ref{fig:conway}, depending on the parity of $n$.
  Each $a_i$ is a non-zero integer counting the number of crossings in the $i$th twist region; the sign of $a_i$ is determined by both the crossing direction and the position of the crossing, following Figure \ref{fig:crossings}.  
  The class of 2-bridge knots can be completely described by the set of all finite sequences of nonzero integers giving positive or negative twists in alternating twist regions.  Note that the number of crossings in the knot diagram is equal to the sum of the absolute values of the $a_i$.

\begin{figure}[ht]
\centering
\begin{tikzpicture}
\draw (1,-1) -- (11,-1);
\draw (1,2) arc (90:270:.5);
\draw (1,0) arc (90:270:.5);
\draw (11,0) arc (-90:90:.5);
\draw (11,-1) arc (-90:90:1.5);
    \foreach \x/\y in {1/0,3/1,7/0,9/1} {
        \draw (\x,\y) -- (\x+.25,\y) -- (\x+.5,\y+.25);
        \draw (\x,\y+1) -- (\x+.25,\y+1) -- (\x+.5,\y+.75);
        \draw (\x+1.5,\y+.25) -- (\x+1.75,\y) -- (\x+2,\y);
        \draw (\x+1.5,\y+.75) -- (\x+1.75,\y+1) -- (\x+2,\y+1);
    }
    \foreach \x/\y in {1/2,3/0,7/2,9/0} {
        \draw (\x,\y) -- (\x+2,\y);
    }
    \foreach \x/\y in {6/2,6/0,6/1} {
        \draw (\x,\y) node {$\ldots$};
    }
\draw (2,.5) node {($a_1$)};
\draw (4,1.5) node {($a_2$)};
\draw (8,.5) node {($a_{n-1}$)};
\draw (10,1.5) node {($a_n$)};

\draw (1,3) -- (11,3);
\draw (1,6) arc (90:270:.5);
\draw (1,4) arc (90:270:.5);
\draw (11,5) arc (-90:90:.5);
\draw (11,3) arc (-90:90:.5);
    \foreach \x/\y in {1/4,3/5,7/5,9/4} {
        \draw (\x,\y) -- (\x+.25,\y) -- (\x+.5,\y+.25);
        \draw (\x,\y+1) -- (\x+.25,\y+1) -- (\x+.5,\y+.75);
        \draw (\x+1.5,\y+.25) -- (\x+1.75,\y) -- (\x+2,\y);
        \draw (\x+1.5,\y+.75) -- (\x+1.75,\y+1) -- (\x+2,\y+1);
    }
    \foreach \x/\y in {1/6,3/4,7/4,9/6} {
        \draw (\x,\y) -- (\x+2,\y);
    }
    \foreach \x/\y in {6/6,6/4,6/5} {
        \draw (\x,\y) node {$\ldots$};
    }
\draw (2,4.5) node {($a_1$)};
\draw (4,5.5) node {($a_2$)};
\draw (8,5.5) node {($a_{n-1}$)};
\draw (10,4.5) node {($a_n$)};

\end{tikzpicture}
\caption{\label{fig:conway} Conway's normal form $C(a_1,a_2,\ldots,a_{n-1},a_n)$ \cite{Con} for 2-bridge knots, depending on the parity of the number of twist regions.  For space considerations, they are drawn horizontally here, but they are usually depicted vertically to give exactly two each of local maxima and minima.}
\end{figure}
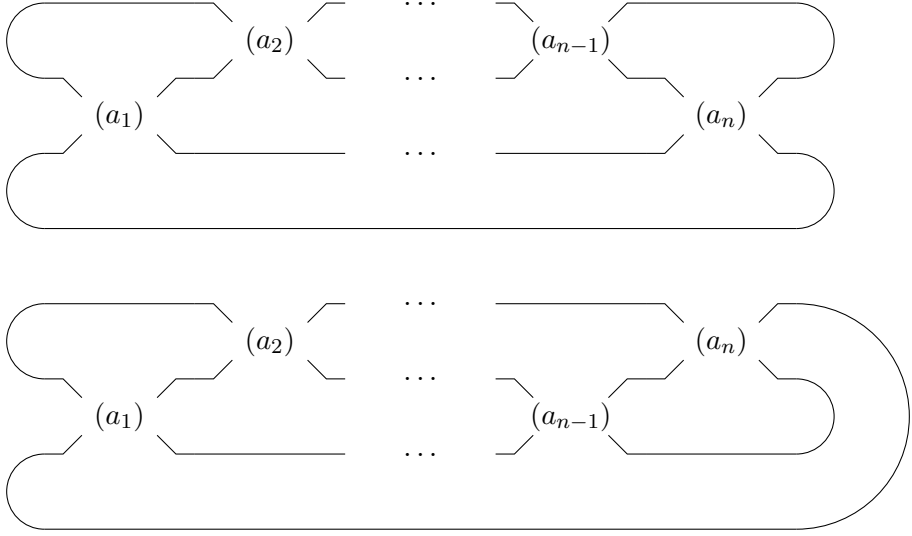

\begin{figure}[ht]
\centering
\begin{tikzpicture}

{
\foreach \x/ \y in {5/0, 6/0,8/0}
    {
    \draw (\x,\y) -- (\x+1,\y+1);
    \draw[color=white, line width=10] (\x+1,\y) -- (\x,\y+1);
    \draw (\x+1,\y) -- (\x,\y+1);
    }
\draw (7.5,.5) node {$\dots$};

\draw (4.5,.5) node {,};

\foreach \x/ \y in {0/0,1/0,3/0}
    {
    \draw (\x+1,\y) -- (\x,\y+1);
    \draw[color=white, line width=10] (\x,\y) -- (\x+1,\y+1);
    \draw (\x,\y) -- (\x+1,\y+1);
    }
\draw (2.5,.5) node {$\dots$};
}

\draw (-1.25,.75) node {positive $a_{2i+1}$};
\draw (-1.25,.25) node {negative $a_{2i}$};
\draw (10.25,.75) node {positive $a_{2i}$};
\draw (10.25,.25) node {negative $a_{2i+1}$};
\end{tikzpicture}
\caption{\label{fig:crossings} The signs of crossings as determined by the $a_i$ notation of Conway's normal form \cite{Con}.}
\end{figure}
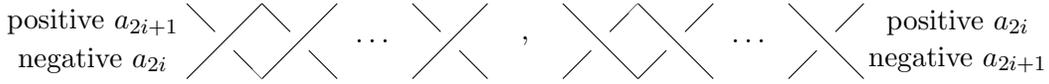

Schubert \cite{Sch} used continued fraction expansions to denote any 2-bridge knot by some rational number:
$$[a_1,a_2,\ldots,a_n]:=a_1+\frac{1}{a_2+\frac{1}{\ldots+\frac{1}{a_{n-1}+\frac{1}{a_n}}}}=\frac{\alpha}{\beta}.$$
Thus they are also called \emph{rational knots}.  The following theorem gives a way to distinguish 2-bridge knots by their rational numbers.

\begin{theorem}
\label{thm:schubert}
\cite{Sch} 
The rational knots $\frac{\alpha}{\beta}$ and $\frac{\alpha'}{\beta'}$ are isotopic if and only if
$\alpha'=\alpha$ and $\beta'\equiv\beta^{\pm1}$ mod $\alpha$.
\end{theorem}

In the next Subsection \ref{subsec:Cheb}, symmetry of our model will allow us reverse the order of the twist regions, and so we are interested in the following result as well:

\begin{theorem}
\label{thm:palindrome} (Palindrome Theorem, see for example \cite{KauffLamb:class}) 
The rational knots $\frac{\alpha}{\beta}=[a_1,\ldots,a_n]$ and $\frac{\alpha'}{\beta'}=[a_n,\ldots,a_1]$ are isotopic if and only if
$\alpha'=\alpha$ and $\beta\beta'\equiv(-1)^{n+1}$ mod $\alpha$.
\end{theorem}

Finally, the \emph{mirror image} of a knot $K$ is the knot $K'$ whose diagram is obtained by changing all  the crossings of a diagram of $K$.  It is known that the mirror image of a 2-bridge knot $K$ with fraction $\frac{\alpha}{\beta}$ is that with fraction $\frac{\alpha}{-\beta}$.

Thus, for a knot $K$ given by the fraction $\frac{\alpha}{\beta}$ and any $\beta'$ such that $\beta\beta'\equiv1\mod\alpha$, we will consider the set
\begin{equation}
\label{set:S}
\mathcal{F}(K):=\left\{\frac{\alpha}{\beta},\frac{\alpha}{-\beta},\frac{\alpha}{\beta'},\frac{\alpha}{-\beta'} \text{ }\bigg| \text{ for all }\beta'\text{ satisfying }\beta\beta'\equiv1\mod\alpha\right\}.
\end{equation}

\subsection{Chebyshev knots}
\label{subsec:Cheb}

Here we introduce a knot model developed by Koseleff and Pecker \cite{KosPec, KosPec3, KosPec4} that will be used for our new random knot model in the following sections.

\begin{definition}
A (long) knot is a \emph{Chebyshev knot} $T(a,b,c)$ if it admits a one-to-one parametrization of the form $x = T_a(t)$; $y = T_b(t)$; $z = T_c(t + \varphi)$, where $T_n(\cos t')=\cos(nt')$ is the $n$th Chebyshev polynomial of the first kind, $t\in\mathbb{R}$, $a, b, c\in\mathbb{Z}$, and $\varphi$ is a real constant.
\end{definition}

These are polynomial analogues of \emph{Lissajous knots}, where each coordinate is parametrized by some $\cos(at+\varphi)$.  Not all knots are Lissajous (see for example \cite{BDHZ}); on the other hand,  for Chebyshev knots we have:
\begin{proposition}
\label{thm:bridge}
\cite[Proposition 9]{KosPec3}
Let $K$ be a knot with $br(K)$ its bridge index. Let $m\geq br(K)$ be an integer.
Then $K$ has a projection which is a Chebyshev curve $x = T_{2m-1}(t)$, $y =T_b(t)$, where $b\equiv 2$ (mod $2(2m-1)$).
\end{proposition}
Koseleff and Pecker use this result to show their main Theorem 3 that every knot has a projection that is a Chebyshev plane curve.

\renewcommand{\thesubfigure}{}
\begin{figure}[ht]
\centering
\subfigure{
}
\hspace{2em}
\subfigure{
\begin{tikzpicture}[scale=.75]
    \foreach \i in {0,...,7} {
        \draw [very thin,dashed,lightgray] (\i,1) -- (\i,4);
    }
    \foreach \i in {1,...,4} {
        \draw [very thin,dashed,lightgray] (0,\i) -- (7,\i);
    }
\draw[thick](-.25,.75) -- (3,4) -- (6,1) -- (7,2) -- (5,4) -- (2,1) -- (0,3) -- (1,4) -- (4,1) -- (7,4) -- (7.25,4.25);
\draw[color=white](-.25,-.25) -- (0,0);
\draw[color=white](7,5) -- (7.25,5.25);
\end{tikzpicture}
}
\caption{\label{fig:Cheb} A Chebyshev knot yielding a (crossingless-)billiard table diagram $T(a,b)$.}
\end{figure}
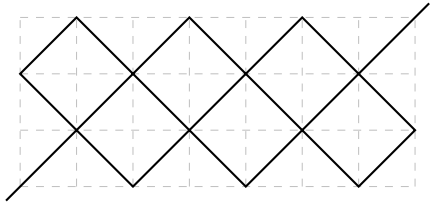

Jones and Przytycki \cite{JonPrz} showed that Lissajous knots are exactly the same set as billiard knots in a cube, where \emph{billiard knots} are periodic trajectories on rational polygonal billiards in a domain.  For Chebyshev knots, Koseleff and Pecker \cite[Proposition 2.6]{KosPec3} prove that there is an even easier visualization.

Consider a rectangular billiard table with width $a$ and length $b$ and the trajectory of a billiard ball shot from the lower left corner at slope 1 as in Figure \ref{fig:Cheb}.  Then any Chebyshev knot $T(a,b,c)$ with any choice of $c\in\mathbb{N}$ and $\varphi\in\mathbb{R}$ can be realized as the \emph{billiard table diagram} $T(a,b)$ with some choice of over- and under-crossing information for the $n=\frac{1}{2}(a-1)(b-1)$ crossings.


Instead of considering all $c\in\mathbb{N}$ and $\varphi\in\mathbb{R}$, in this paper we consider just the set of $\{+,-\}^n$ possibilities for the signs of the $n$ crossings of the diagram for a fixed $T(a,b)$.

Let us consider a given string in $\{+,-\}^n$ as a word of length $n$ in the alphabet $\{+,-\}$.  For some letter $u\in\{+,-\}$, we will use the notation $\bar{u}$ to mean the \emph{negative} or \emph{opposite} letter; the negative $\bar{w}$ of a word is defined similarly.  For example, with $w=++-$ we have $\bar{w}=--+$.

The $i$th letter in the string gives the sign of the $i$th crossing when ordered in columns from left to right, proceeding upwards in each column.  Positive crossings are shown on the left of Figure \ref{fig:crossings}; negative crossings appear on the right.

\subsection{Chebyshev knots with $a=3$}
\label{subsec:a2}

The scope of this paper is only the case where $a=3$, giving $n=b-1$ crossings in the billiard table diagram $T(a,b)$.  Note that this class of knots includes all 2-bridge knots and the unknot (the only 1-bridge knot).  The results that follow pertain only to this case, although the authors are aware of some obvious generalizations for larger bridge index.

We also ignore the cases where $3|b$, as these give uninteresting two-component 2-bridge links instead of knots.

The interested reader may find many examples of Chebyshev knot diagrams for 2-bridge knots on Pierre-Vincent Koseleff's website:\\
\verb=http://www.math.jussieu.fr/~koseleff/knots/kindex1.html= 

We now introduce \emph{internal and external reduction moves} and establish equivalence between these and definitions in the continued fraction language. Thereafter we use these notions to prove in Theorem \ref{thm:noMoves} that there is just a small probability for knots to occur without any reduction moves. In fact, we prove a stronger result later in Corollary \ref{cor:noint}, following the Main Theorem \ref{thm:main}.

\begin{definition}
Let $T^{(n)}_K$ denote the event of a knot $K$ occurring in $T(3,n+1)$, when each of the $n$ crossings are chosen independently to be positive or negative with equal probability. In other words, the crossings are chosen uniformly at random.
\end{definition}

\begin{definition}
\label{def:int}
Suppose a string of crossings of length $n$ contains the substring $+++$ or $---$.  An \emph{internal reduction move} is the deletion of the substring to result in a string of length $n-3$.  This move may be written as $\Int^{(n)}$.
\end{definition}

\begin{proposition}
\label{prop:int}
Performing an internal reduction move does not change the knot type.
\end{proposition}

\begin{proof}
This move is a half-twist rotation along the ``$b$'' axis, removing three crossings as in Figure \ref{fig:IntRed}.  All crossings that appear afterwards have their positions exchanged without changing signs:  for example, a positive crossing between the first two strands becomes a positive  crossing between the second two strands.  This coincides with the change of positions caused by the removal of an odd number of crossings.

Since there is but a single strand joining the two ends of the billiard table diagram, this half-twist does not propogate.
\end{proof}


\renewcommand{\thesubfigure}{}
\begin{figure}[ht]
\centering
\subfigure{
\begin{tikzpicture}[scale=.6]
    \foreach \i in {0,...,4} {
        \draw [very thin,dashed,lightgray] (\i,1) -- (\i,4);
    }
    \foreach \i in {1,...,4} {
        \draw [very thin,dashed,lightgray] (-.5,\i) -- (4.5,\i);
    }

\draw[thick,color=red] (-.5,2.5) -- (1,4) -- (4,1) -- (4.5,1.5);

\fill[color=white] (3,2) circle (5pt);

\draw[thick](4.5,3.5) -- (2,1) -- (-.5,3.5);

\fill[color=white] (1,2) circle (5pt);
\fill[color=white] (2,3) circle (5pt);

\draw[ultra thick,color=blue](-.5,1.5) -- (0,1) -- (3,4) -- (4.5,2.5);

\end{tikzpicture}
}
\hspace{2em} $\rightsquigarrow$ \hspace{2em}
\subfigure{
\begin{tikzpicture}[scale=.6]
    \foreach \i in {0,...,4} {
        \draw [very thin,dashed,lightgray] (\i,1) -- (\i,4);
    }
    \foreach \i in {1,...,4} {
        \draw [very thin,dashed,lightgray] (-.5,\i) -- (4.5,\i);
    }
\draw[thick,color=red] (-.5,2.5) -- (.5,3.5) -- (3.5,3.5) -- (4,4) -- (4.5,3.5);


\draw[thick](4.5,1.5) -- (3.5,2.5) -- (.5,2.5) -- (-.5,3.5);


\draw[ultra thick,color=blue](-.5,1.5) -- (0,1) -- (.5,1.5) -- (3.5,1.5) -- (4.5,2.5);

\end{tikzpicture}
}
\caption{\label{fig:IntRed}An example of an internal reduction move, where the $i$th letter in the string gives the sign of the $i$th crossing, ordered from left to right.}
\end{figure}
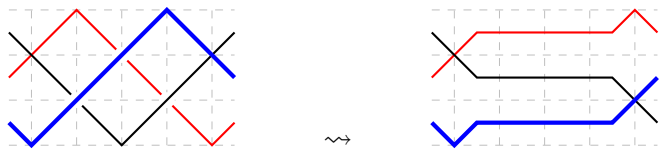

\begin{definition}
\label{def:ext}
Suppose a string of crossings of length $n$ begins with the substring $++*$ or $--*$ or ends with the substring $*++$ or $*--$, where $*$ could be either $+$ or $-$.  An \emph{external reduction move} is the deletion of the substring to result in a string of length $n-3$.  This move may be written as $\Ext^{(n)}$.
\end{definition}

\begin{proposition}
\label{prop:ext}
Performing an external reduction move does not change the knot type.
\end{proposition}

\begin{proof}
This move is a Reidemeister II move followed by a Reidemeister I move as in Figure \ref{fig:ExtRed}.

An external reduction move performed at the beginning must be followed by a complete half-twist rotation as above so that the first crossing appears between the first two strands.
\end{proof}

Note that some of the external reduction moves are examples of the internal reduction move.

\renewcommand{\thesubfigure}{}
\begin{figure}[ht]
\centering
\subfigure{
\begin{tikzpicture}[scale=.6]
    \foreach \i in {0,...,4} {
        \draw [very thin,dashed,lightgray] (\i,1) -- (\i,4);
    }
    \foreach \i in {1,...,4} {
        \draw [very thin,dashed,lightgray] (0,\i) -- (4.5,\i);
    }
\draw[thick,color=red](4.5,3.5) -- (2,1) -- (0,3) -- (1,4) -- (4,1) -- (4.5,1.5);

\fill[color=white] (1,2) circle (5pt);
\fill[color=white] (2,3) circle (5pt);

\draw[ultra thick,color=blue](-.25,.75) -- (3,4) -- (4.5,2.5);

\end{tikzpicture}
}
\hspace{2em} $\rightsquigarrow$ \hspace{2em}
\subfigure{
\begin{tikzpicture}[scale=.6]
    \foreach \i in {0,...,4} {
        \draw [very thin,dashed,lightgray] (\i,1) -- (\i,4);
    }
    \foreach \i in {1,...,4} {
        \draw [very thin,dashed,lightgray] (0,\i) -- (4.5,\i);
    }
\draw[thick,color=red](4.5,3.5) -- (3,2) -- (4,1) -- (4.5,1.5);


\draw[ultra thick,color=blue](2.75,4.25) -- (4.5,2.5);

\end{tikzpicture}
}
\caption{\label{fig:ExtRed}An example of an external reduction move, where the $i$th letter in the string gives the sign of the $i$th crossing, ordered from left to right.}
\end{figure}
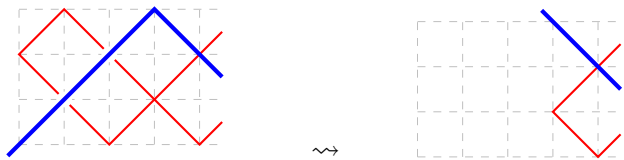

These definitions coincide with a  notion of Koseleff and Pecker:

\begin{definition}
\label{def:1regular}
\cite[Definition 2.4]{KosPec4}A continued fraction $[a_1,a_2,\ldots,a_n]$ is \emph{1-regular} if it has the following
properties:
\begin{itemize}
	\item[(i.)] $a_i \neq 0$
	\item[(ii.)] $a_{n-1}a_n > 0$, and
	\item[(iii.)] if $a_ia_{i+1} < 0$, then $a_{i+1}a_{i+2} > 0$ for $i=1,\ldots,n-2$.
\end{itemize}
\end{definition}

\begin{remark}
\label{rem:1regular}
In Definition \ref{def:1regular} above, property (ii.) is equivalent to there being \emph{NO external reduction moves} at the end on the right, and property (iii.) is equivalent to there being \emph{NO internal reduction moves}.
\end{remark}

Assuming $a_1=1$ throughout, this definition is used to show uniqueness of a 1-regular continued fraction for a given rational number:

\begin{theorem}
\label{thm:1regular}
\cite[Theorem 3.1]{KosPec4}
Let $\frac{\alpha}{\beta}$ be a rational number. There is a unique 1-regular continued fraction such that $a_i = \pm1$ and $\frac{\alpha}{\beta}= [a_1,a_2,\ldots,a_n]$.  
Furthermore, $\frac{\alpha}{\beta} > 1$ if and only if $a_2 = 1$.
\end{theorem}

\begin{remark}
\label{rem:alphabeta1}
The last property $a_2=1$, together with the supposition throughout that $a_1=1$, is equivalent to there being \emph{NO external reduction move} at the end on the left.  Therefore this is the condition we want on the continued fraction, together with being 1-regular.  As such, we will take $|\beta|\leq\alpha$.
\end{remark}

This will be used to prove Theorem \ref{thm:noMoves}, 
 along with one more idea.

\begin{definition}
\label{def:length1}
\cite[almost Definition 3.3]{KosPec4}
Let $\frac{\alpha}{\beta}>1$ be the 1-regular continued fraction $[a_1,\ldots,a_n]$, $a_i = \pm1$. We will denote its \emph{length} (or \emph{reduced length}) $n$ by $\ell(\frac{\alpha}{\beta})$. Note that $\ell(\frac{\alpha}{-\beta})=\ell(\frac{\alpha}{\beta})$.
\end{definition}

\begin{remark}
\label{rem:length}
The original definition just has $\frac{\alpha}{\beta}>0$, but as stated above, we need to ensure that there is no external reduction move at the end on the left.
\end{remark}

Following this definition, Koseleff and Pecker provide several results on the topic, including two that will appear in the proof of Theorem \ref{thm:noMoves} below.

\begin{remark}
\label{rem:PGP}
The first of these involves some new notation, $\frac{\alpha}{\beta}= PGP(\infty)$, which is equivalent to the condition that there are no external reduction moves.

We also let $cn(K)$ be the crossing number of the knot $K$.
\end{remark}

\begin{proposition}
\label{prop:3.8}
\cite[abbreviated Proposition 3.8]{KosPec4}
Let $\alpha > \beta > 0$ and consider $\frac{\alpha}{\beta}= PGP(\infty)$ and $N = cn (\frac{\alpha}{\beta})$. Let $\beta'$
be such that $0 < \beta' < \alpha$ and $\beta\beta' \equiv (-1)^{N-1} \mod \alpha$. Then we [...] also have
$\ell(\frac{\alpha}{\alpha - \beta}) + \ell(\frac{\alpha}{\beta}) = 3N - 2$ and $\ell(\frac{\alpha}{\beta'}) = \ell(\frac{\alpha}{\beta})$.
\end{proposition}

\begin{proposition}
\label{prop:4.3}
\cite[abbreviated Proposition 4.3 (2.)]{KosPec4}
Let $K$ be a two-bridge knot with crossing number $N$.  There exists $\frac{\alpha}{\beta}>1$ such that $K$ can be represented by $(\pm\frac{\alpha}{\beta})$ and $\ell(\frac{\alpha}{\beta})<\frac{3}{2}N-1$.   [...]
\end{proposition}

At last we arrive at our first theorem.

\begin{theorem}
\label{thm:noMoves}
For any knot $K$, we have
$$P(T^{(n)}_K|\not\exists\Int^{(n)}, \not\exists\Ext^{(n)})=0$$
for all $n$ except for $\ell_0$ and $\ell_1$.  In these cases, we have
$$P(T^{(n)}_K|\not\exists\Int^{(n)}, \not\exists\Ext^{(n)})=
\begin{cases}
\displaystyle\frac{2}{2^n} & \text{ when $r(\ell)=2$}\\
& \\
\displaystyle\frac{4}{2^n} & \text{ when $r(\ell)=4$},
\end{cases}
$$
where $r(\ell)$ is the number of ways of to write $K$ in the length $\ell$.

In particular, this means that the limit as $n$ approaches infinity is also zero.
\end{theorem}

We reiterate that a stronger result appears in Corollary \ref{cor:noint}. The bulk of the proof of the present result relies on the following lemma.

\begin{lemma}
\label{prop:lengths}
Let $K$ be a knot with bridge index at most 2, written as some Chebyshev ($|a_i|=1$) continued fraction expansion with no reduction moves.  Then there are exactly two reduced lengths $\ell$ for $K$.  Furthermore, these lengths are $\ell_0\equiv 0 \mod 3$ and $\ell_1\equiv 1\mod 3$.
\end{lemma}

\begin{proof}
Recall from Equation \ref{set:S} the set $\mathcal{F}(K):=\{\frac{\alpha}{\pm\beta},\frac{\alpha}{\pm\beta'} |$ for all $\beta'$ satisfying $\beta\beta'\equiv1\mod\alpha\}$.  From Remark \ref{rem:alphabeta1}, we may fix a particular $0 <\beta'<\alpha$ such that $\beta\beta'\equiv(-1)^{n+1}$ mod $\alpha$ and need only consider the set $\mathcal{F}'(K):=\{\frac{\alpha}{\pm\beta},\frac{\alpha}{\pm\beta'},\frac{\alpha}{\pm(\alpha-\beta)},\frac{\alpha}{\pm(\alpha-\beta')}\}$ of size eight.

Also by Definition \ref{def:length1} we need not distinguish mirrors $\frac{\alpha}{\pm\beta}$, and so we consider the set\\ $\{\frac{\alpha}{\beta},\frac{\alpha}{\beta'},\frac{\alpha}{\alpha-\beta},\frac{\alpha}{\alpha-\beta'}\}$ of size four.

Let $N=cn(\frac{\alpha}{\beta})$ be the crossing number of $K$, and note that $(\alpha-\beta)^{-1}$ can be written as $(\alpha-\beta')$ for the appropriate $\beta'$ as above. Then by Proposition \ref{prop:3.8} \cite[Proposition 3.8]{KosPec4}, we need not distinguish inverses since $\ell(\frac{\alpha}{(-1)^{N-1}\beta^{-1}})=\ell(\frac{\alpha}{\beta})$, and so we consider the set $\{\frac{\alpha}{\beta},\frac{\alpha}{\alpha-\beta}\}$ of size two.

Following the argument above starting with the set $\mathcal{F}'(K)$, we have $r(\ell)=4$.  This holds except when some of the fractions in this set coincide due to the Palindrome Theorem \ref{thm:palindrome}; in this case we have $r(\ell)=2$.

By Proposition \ref{prop:4.3} \cite[Proposition 4.3 (2.)]{KosPec4}, one of these has length less than $\frac{3}{2}N-1$.  Again by Proposition \ref{prop:3.8} \cite[Proposition 3.8]{KosPec4}, we have $\ell(\frac{\alpha}{\beta})+\ell(\frac{\alpha}{\alpha-\beta})=3N-2$.  Since the two lengths must sum to $1 \mod 3$, and since neither length can be $2\mod 3$ since otherwise it would give a two-component link and not a knot, exactly one of these lengths must be $0\mod 3$ and the other $1\mod 3$, as desired.
\end{proof}


The proof of the theorem now follows directly.

\begin{proof}
First suppose the knot $K$ has bridge index of three or greater; then it can appear in no $T(3,b)$ by the construction of the model.

Now suppose that the knot $K$ has bridge index at most 2.  Then by Lemma \ref{prop:lengths}, there are just two reduced lengths $\ell_0$ and $\ell_1$ for $K$, and each of these lengths has at most four rationals associated to it, as described in the proof above.

For $K$ a 2-bridge knot, by Theorem \ref{thm:1regular}, there are unique continued fractions associated with these four rational numbers, and thus there are at most four Chebyshev billiard table diagrams with no reduction moves associated  with $K$.

For $K$ the unknot, the lengths are $\ell_0=0$ giving a single empty word and $\ell_1=1$ giving the two words $+$ and $-$.
\end{proof}

\section{Main Definitions}
\label{sec:maindefs}
In this section, we set up some notation and definitions first and then prove general results that will prepare us to state and prove our Main Theorem \ref{thm:main} in the next Section \ref{sec:mainresults}.

\begin{definition}
\label{def:xy}
Let $\Int_j^{(n)}$ be the occurrence of an internal reduction move in positions $j$, $j+1$, and $j+2$ in a string of length $n$.

Let $S^{(n)}=S^{(n)}_K$ be the set $\{T_K^{(n)}\}$, and let us establish a descending filtration up to finite index of the set $S_1^{(n)}:= S^{(n)}\amalg S^{(n)}$ by $S_2^{(n)}:=S^{(n)}=:S_3^{(n)}$ and $S^{(n)}_i=\{T_K^{(n)},\not\exists \Int^{(n)}_j \text{ for }j< i-2\}$ for $4\leq i \leq n+1$.  Note that here $S^{(n)}_{i+1}\subseteq S^{(n)}_{i}$ for $i\leq n$.

Let $x_i^{(n)}:=|D_i^{(n)}|:=|S_i^{(n)}\setminus S_{i+1}^{(n)}|$ be the cardinality of the difference between these two sets for $1\leq i \leq n$, and introduce $y_i^{(n)}:=x_i^{(n)}-x_{i+1}^{(n)}$ for $1 \leq i \leq n-1$.  These are defined for some fixed $K$, which is assumed to be understood from now on.
\end{definition}

The above definitions are motivated by the following lemma.
\begin{lemma}
\label{lem:xs}
Following the notation above, $x^{(n)}_i=|S^{(n-3)}_{i-2}|$ for $3\leq i\leq n$ and $n\geq 4$.
\end{lemma}

\begin{proof}
From Definition \ref{def:xy}, we see that
$$x_{i}^{(n)}=|D^{(n)}_i|=|\{T^{(n)}_K,\exists \Int^{(n)}_{i-2},\not\exists \Int^{(n)}_j\text{ for } j<i-2\}|\text{ for } 3\leq i\leq n\text{ and }n\geq 4.$$
The following idea for relating cardinalities of the two sets $D^{(n)}_i$ and $S^{(n-3)}_{i-2}$ will be used again, so we detail the steps for this time even though the method is not hard. We also remind the reader that  $S^{(n-3)}_{i-2}=\{T_K^{(n-3)},\not\exists \Int^{(n-3)}_j \text{ for }j< i-4\}$ for $6\leq i \leq n$, that  $S^{(n-3)}_{2}=S^{(n-3)}_{3}=S^{(n-3)}$ corresponding to $i=4,5$,  respectively, and that $S_1^{(n-3)}=S^{(n-3)}\amalg S^{(n-3)}$ corresponding to $i=3$.

Let $s\in S^{(n-3)}_{i-2}$ be arbitrary with $i\geq 4$. It can be written as
$s=w_1 q w_2,$
where  the word $w_1$ is of length $i-4$ and has no $\Int$ move anywhere, the letter $q$ is in the $(i-3)$rd position, and the word $w_2$ is the remaining part. We are interested in adding an $\Int$ move to $s$ so as to obtain different $s'\in D^{(n)}_i$. Clearly
$$s'=w_1 \,q\, \bar{q}\, \bar{q}\, \bar{q}\, w_2\in D^{(n)}_i,$$
where $\bar{q}$ is $q$'s  negative (or opposite). Note that there is exactly one way to do this because any element of $D^{(n)}_i$ should have an $\Int$ move starting at the $(i-2)$nd position and {\it no} such moves before it. Also, different $s\in S^{(n-3)}_{i-2}$ give rise to different $s'\in D^{(n)}_i$. Thus, $|D^{(n)}_i|\geq |S^{(n-3)}_{i-2}|$.

Conversely, on performing the $\Int^{(n)}_{i-2}$ move, it is straightforward to see that any arbitrary $s'\in D^{(n)}_i$ must be obtained from some $s\in S^{(n-3)}_{i-2}$ in this manner, that is, by the addition of an $\Int$ move.  We conclude that $|D^{(n)}_i|=|S^{(n-3)}_{i-2}|$.

When $i=3$,
$$x_{3}^{(n)}=|D^{(n)}_3|=|\{T^{(n)}_K,\exists \Int^{(n)}_{1}\}|\text{ and }S_1^{(n-3)}=S^{(n-3)}\amalg S^{(n-3)}.$$
Taking an arbitrary $s\in S^{(n-3)}$, it is easy to see that there are now two ways to add an $\Int$ move to the left end of $s$ to obtain different $s'\in D^{(n)}_3$. The argument for relating the cardinalities then proceeds exactly along the same lines as in the case detailed first. The lemma is thus fully proven.
\end{proof}

There are a few other basic properties that are also of interest, and we have collected them as the next block.

\begin{property}
\label{property:xy}
Recall that $x_i^{(n)}$ is defined for $i\leq n$ and that $y_i^{(n)}$ is defined for $i\leq n-1$.  For appropriate values of $n$ unless otherwise specified, we have

\begin{center}
\begin{tabular}{ll}
\multicolumn{2}{l}{(Sn) $|S_{n+1}^{(n)}|=x_{n+3}^{(n+3)}$ for $n\geq 1$}\\
(X1) $x_1^{(n)}=|S^{(n)}|$ 		& (Y1) $y_1^{(n)}=|S^{(n)}|$					 \\
(X2) $x_2^{(n)}=0$ 						& (Y2) $y_2^{(n)}=-2x_4^{(n)}$				 \\
(X3) $x_3^{(n)}=2x_4^{(n)}$ 	& (Y3) $y_3^{(n)}=x_4^{(n)}$ 					 \\
(X4) $x_4^{(n)}=x_5^{(n)}$		& (Y4) $y_4^{(n)}=0$ 		 
\end{tabular}
\end{center}
\end{property}

\begin{proof}
(X1), (X2), and (Y1) are direct consequences of Definition \ref{def:xy}.  Also, (Sn) follows directly from Lemma \ref{lem:xs}. From Lemma \ref{lem:xs} and Definition \ref{def:xy} again, we have that 
\begin{align*}
x^{(n)}_3 & =|S^{(n-3)}_1|=2|S^{(n-3)}|,\\
x^{(n)}_4 & =|S^{(n-3)}_2|=|S^{(n-3)}|,\text{ and}\\
x^{(n)}_5 & =|S^{(n-3)}_3|=|S^{(n-3)}|.
\end{align*}
(X3) and (X4) are thus done. Relying on Definition \ref{def:xy}, (Y2), (Y3), and (Y4) are also then proven.

Notice the implications between the properties: properties (X1) and (X2) together imply (Y1); properties (X2) and (X3) together imply (Y2); properties (X3) and (X4) together imply (Y3); and property (X4) implies (Y4).
\end{proof}

From Definition \ref{def:xy}, it is clear that for $1\leq i \leq n$, one can write
\begin{equation}
\label{eq:fact}
x_i^{(n)}=x_n^{(n)}+\sum_{j=i}^{n-1}y_j^{(n)}.
\end{equation}

We next observe that these also satisfy the following lemma, which is actually a simple corollary of Lemma \ref{lem:xs}. We state it as a separate lemma because it will be used repeatedly in the sequel in conjunction with Equation \ref{eq:fact}.
\begin{lemma}[$xy$ Lemma]
\label{lem:XY}
Following the notation above, $y_i^{(n)}=x_{i-2}^{(n-3)}$ for $3\leq i \leq n-1.$
\end{lemma}

\begin{proof}
From Definition \ref{def:xy}, $y^{(n)}_i=x^{(n)}_i-x^{(n)}_{i+1}$. Since $3\leq i \leq n-1$, thus requiring $n\geq 4$, we can use Lemma \ref{lem:xs} to conclude that $y_i^{(n)}$ is just $|S^{(n-3)}_{i-2}|-|S^{(n-3)}_{i-1}|$. Going back to Definition \ref{def:xy}, we see that the last expression is exactly $x^{(n-3)}_{i-2}$.
\end{proof}

Equation \ref{eq:fact} and the $xy$ Lemma \ref{lem:XY} are two results that we return to time and again to prove our main theorem. 

Now that we have established this theory for $x_i^{(n)}$ for $n\geq 1$, we wish to extend it to all $i\leq n\in\mathbb{Z}$ in order to prove Lemma \ref{lem:binom}, which will then be used to prove the Main Theorem \ref{thm:main}.  
  Extending the theory is done  iteratively, starting with $i=0$.

\begin{definition}
\label{def:i0}
For each $i\leq 0$, iteratively set $x_i^{(n)}:= y_{i+2}^{(n+3)}$ following the $xy$ Lemma \ref{lem:XY}.  Then set $y_i^{(n)}:= x_i^{(n)}-x_{i+1}^{(n)}$ as above.
\end{definition}

Note then, for example, that for $n=0$ we have $x_0^{(0)}=y^{(3)}_2=x_2^{(3)}-x_3^{(3)}=-x_3^{(3)}$, and for $n\geq 1$, we have $x_0^{(n)}=y^{(n+3)}_2=-2x_4^{(n+3)}$ by Property \ref{property:xy} (Y2).

\begin{remark}
\label{rem:xni}
 This definition is motivated by the fact that for $i\leq i'-1\leq 0$, the new $x_i^{(n)}$ can only be defined in terms of the $x_{i'}^{(n)}$ using the $xy$ Lemma \ref{lem:XY}, which is central to proving our Main Theorem \ref{thm:main}.  Thus these are uniquely defined.  The motivation for defining the $y_i^{(n)}$ as above is so that we have Equation \ref{eq:fact}, which is also an important ingredient  in proving our Main Theorem \ref{thm:main}.
\end{remark}

\section{Main Results}
\label{sec:mainresults}

With all the notations and definitions out of the way, in the present section we turn to deriving the main result. The first step in that direction is the lemma given below.

\begin{lemma}
\label{lem:binom}
For $i\leq n\in\mathbb{Z}$, the $x_i^{(n)}$ are given by
\begin{equation*}
x_i^{(n)}=\sum_{j=0}^{n-i} \binom{n-i}{j} x_{n-3j}^{(n-3j)}.
\end{equation*}
\end{lemma}

\begin{proof}
We use induction on $d:= n-i$. Since $\binom{0}{0}=1$, the base case $d=0$ is trivial. Assume the result holds for all positive integers up to $ d-1$. Let us verify that the equation is correct for $ d$ as well. We have
$$x^{(n)}_i=x^{(n)}_n+\sum_{j=i}^{n-1}y^{(n)}_j=x^{(n)}_n+\sum_{j=i}^{n-1}x^{(n-3)}_{j-2},$$
where we used Equation \ref{eq:fact} in the first equality and the $xy$ Lemma \ref{lem:XY} in the second, along with their extensions from  Definition \ref{def:i0}.

Since $(n-3)-(j-2)\leq  d-1$, we have by the induction hypothesis that
$$x^{(n)}_i = x^{(n)}_n+\sum_{j=i}^{n-1}\sum_{k=0}^{n-j-1}\binom{n-j-1}{k}x^{(n-3k)}_{n-3k}.$$
Exchanging the order of the summations, it is easily seen that we have
\begin{align*}
x^{(n)}_i &= x^{(n)}_n+\sum_{k=0}^{n-i-1}\sum_{j=i}^{n-k-1}\binom{n-j-1}{k}x^{(n-3k)}_{n-3k}\\
&= x^{(n)}_n+\sum_{k=0}^{n-i-1}\binom{n-i}{k+1}x^{(n-3k)}_{n-3k}=\sum_{k=0}^{n-i} \binom{n-i}{k} x_{n-3k}^{(n-3k)}.
\end{align*}
\end{proof}

Finally we are ready to state and prove the Main Theorem \ref{thm:main}, which has three parts. The first gives a simple closed form expression for the probability of any knot $K$ occurring as a string of length $n$ in terms of  the $x^{(i)}_i$. Formulas for these $x^{(i)}_i$ appear as the second and third parts.

\begin{theorem}
\label{thm:main} (Main Theorem)
\begin{enumerate}
\item The probability of a given knot $K$ appearing in $T(3,n+1)$ is
\begin{equation}
\frac{1}{2^n}\left[x_{n+3}^{(n+3)}+\sum_{i=0}^{n-6} \left[\binom{n-5}{i+1}+4\binom{n-5}{i}\right] x_{n-3i}^{(n-3i)} + 4x_{-2n+15}^{(-2n+15)}\right].
\label{eq:maineq}
\end{equation}

\item Let $K$ be a knot with reduced lengths $(\ell=)$ $\ell_0\equiv 0$ and $\ell_1\equiv 1 \mod 3$ and let $r(\ell)$ be the number of ways to write $K$ in that length.  Let $U$ denote the unknot.  Then

\begin{equation}
\label{eq:2}
x_{3m+\ell}^{(3m+\ell)}=
\begin{cases}
4(m-1)+2& \text{for $K=U$ with $m\geq1$ and $\ell=\ell_1=1$}\\
4(m-1)r(\ell)  &	\text{for $K=U$ with $m\geq 2$ and $\ell=\ell_0=0$ and}\\
								&	\text{for $K\neq U$ with $m\geq 2$.}
\end{cases}
\end{equation}
As base cases we have:  
\begin{itemize}
	\item[(i.)] $x^{(1)}_1=x^{(4)}_4$ for any knot $K$;
	\item[(ii.)]  $x^{(3)}_3 =x^{(4)}_4=2$ for $K=U$;
	\item[(iii.)]  $x_{3+\ell}^{(3+\ell)}=r(\ell)$ for $K\neq U$; and
	\item[(iv.)]  $x^{(i)}_{i}=0$ for $K\neq U$ when $3\leq i\leq \ell$.
\end{itemize}

\item  For $6\leq n\in\mathbb{Z}^+$ we have
$$x_{2(6-n)}^{(2(6-n))} = \sum_{j=1}^{n-5} \left[ - \binom{n-5}{j-1} \right] x_{n-3j}^{(n-3j)}.$$

For $5\leq n\in\mathbb{Z}^+$ with $\binom{k}{-1}:=0$ for $k\geq 0$ we have
$$x_{2(4-n)+1}^{(2(4-n)+1)} = \sum_{j=0}^{n-4} \left[ \binom{n-4}{j} - \binom{n-4}{j-1} \right] x_{n-3j}^{(n-3j)}.$$

\end{enumerate}
\end{theorem}

\begin{proof}
{\it Part (1)}:
We have: 
$$|S^{(n)}|  = |S_{n+1}^{(n)}|+ \sum_{i=3}^n x_i^{(n)} = |S_{n+1}^{(n)}|+ \sum_{i=6}^n x_i^{(n)} + 4x_5^{(n)},$$
the first equality following from Definition \ref{def:xy}, and the second from Property \ref{property:xy} (X3) and (X4). In fact, the first equality is merely a statement of the fact that $D^{(n)}_i$ for $3\leq i\leq n$ together with $S^{(n)}_{n+1}$ 
 form a partition of $S^{(n)}$: where an internal reduction move first occurs, if it exists, and when no such move exists.

Property \ref{property:xy} (Sn) and Lemma \ref{lem:binom} now say that the above is just   
\begin{align*}
|S^{(n)}|  &=  x_{n+3}^{(n+3)} + \sum_{i=6}^n \sum_{j=0}^{n-i} \binom{n-i}{j} x_{n-3j}^{(n-3j)} + 4\sum_{j=0}^{n-5} \binom{n-5}{j} x_{n-3j}^{(n-3j)}.
\end{align*}
Interchanging order of summations in the second term, and taking out the case $j=n-5$ in the summation in the third term, we get 
\begin{align*}
|S^{(n)}|  &= x_{n+3}^{(n+3)} + \sum_{j=0}^{n-6} \left[\sum_{i=6}^{n-j} \binom{n-i}{j} + 4\binom{n-5}{j}\right] x_{n-3j}^{(n-3j)} + 4 x_{-2n+15}^{(-2n+15)}\\
&= x_{n+3}^{(n+3)}+\sum_{i=0}^{n-6} \left[\binom{n-5}{i+1}+4\binom{n-5}{i}\right] x_{n-3i}^{(n-3i)} + 4x_{-2n+15}^{(-2n+15)}.
\end{align*}
Finally, the probability is obtained on dividing by the total number of possible binary sequences of length $n$. 

\bigskip

{\it Part (2)}:  
  For the base case (i.) we have for any knot $K$ that $x^{(1)}_1=y^{(4)}_3=x^{(4)}_4$ by the $xy$ Lemma \ref{lem:XY} and Property \ref{property:xy} (Y3).

\medskip

For base case (ii.) where $K=U$, $m=1$, and $\ell=\ell_0=0$, we need to show $x^{(3)}_3=2$. From Definition \ref{def:xy},
\begin{equation}
x^{(3)}_3=|S^{(3)}_3\setminus S^{(3)}_4|=|S^{(3)}\setminus \{T^{(3)}_K,\not\exists\Int^{(3)}_1\}|=|\{T^{(3)}_K,\exists\Int^{(3)}_1\}|=|\{+++,---\}|=2.
\label{eq:x33}
\end{equation}


Next take $m=1$ and $\ell=\ell_1=1$ for the unknot. We want to prove that $x^{(4)}_4=2.$ Again relying on Definition \ref{def:xy},
\begin{equation}
x^{(4)}_4=|S^{(4)}_4\setminus S^{(4)}_5|=|\{T^{(4)}_K,\exists\Int^{(4)}_2,\not\exists\Int^{(4)}_1\}|=|\{-+++,+---\}|=2.
\label{eq:x44}
\end{equation}

\medskip

Now for base case (iii.) we focus on the case $m=1$ for a general knot $K\neq U$ with general $\ell$. Here we have
$$x^{(3+\ell)}_{3+\ell}=|S^{(\ell)}_{1+\ell}|=|\{T^{(\ell)}_K,\not\exists\Int^{(\ell)}\}|=r(\ell),$$
where the first equality uses Property \ref{property:xy} (Sn), the second is from Definition \ref{def:xy}, and the last follows from the definition of $\ell$ and $r(\ell)$.

\medskip

In the final base case (iv.) we need to show that for any knot $K\neq U$, we have $x^{(i)}_i=0$ for $3\leq i\leq \ell$. Observe first that the cases $i=3$ and $i=4$ are immediate from Equation \ref{eq:x33} and Equation \ref{eq:x44}, respectively. For $5\leq i\leq \ell$, we use Property \ref{property:xy} (Sn) to note that $x^{(i)}_i=|S^{(i-3)}_{i-2}|=0$ from definition of the reduced lengths. 

\medskip

We move on to the second case in Equation \ref{eq:2} when $m\geq 2$ for $K=U$ and $\ell=\ell_0=0$ or where $K\neq U$ with a general $\ell$ $(>1)$. We consider $$x^{(3m+\ell)}_{3m+\ell}=|S^{(3(m-1)+\ell)}_{3(m-1)+\ell+1}|=|\{T^{(3(m-1)+\ell)}_K,\not\exists\Int^{(3(m-1)+\ell)}\}|,$$
where the first equality is by Property \ref{property:xy} (Sn) and the second follows from Definition \ref{def:xy}. We would like to show that this equals $4(m-1)r(\ell)$. 

Let the knot $K$ be written in reduced length $\ell$ as the word $w$. Also, let the words $++-$ and $-++$ be denoted by $e$ and $f$, respectively, and recall that $\bar{e}$ and $\bar{f}$ denote the negatives of these words. In this formulation, the cardinality $|S^{(3(m-1)+\ell)}_{3(m-1)+\ell+1}|$ we are after is equal to the number of ways in which to adjoin $e$ and $f$ to $w$ so as to end up with a word of length $3(m-1)+\ell$ with no internal reduction moves. Notice that all of $e\bar{e}$, $\bar{e}e$, $f\bar{f}$, and $\bar{f}f$ result in the introduction of $\Int$ moves and so cannot appear. Therefore, the length $3(m-1)+\ell$  word must be written as
$$e^{m_1}wf^{m_2},\bar{e}^{m_1}wf^{m_2},e^{m_1}w\bar{f}^{m_2},\text{ or }\bar{e}^{m_1}w\bar{f}^{m_2},$$
where $m_1,m_2\in\mathbb{Z}^{\geq0}$ such that $m_1+m_2=m-1$.

First suppose that $m_1$ is not $0$ or $m-1$. Then for each way to write $K$ in reduced length $\ell$, there is the choice between $e$ and $\bar{e}$ and the independent choice between $f$ and $\bar{f}$.  Thus the count here is $4(m-2)r(\ell)$.

Next suppose that $m_1$ is either $0$ or $m-1$. For each way of writing $K$ in reduced length $\ell$, in the first case, there is the choice between $f$ and $\bar{f}$, and in the second case, there is the choice between  $e$ and $\bar{e}$.  Therefore we need to add $4r(\ell)$ to the count.  Summing up, we have:  
$$x^{(3m+\ell)}_{3m+\ell}=4(m-2)r(\ell)+4r(\ell)=4(m-1)r(\ell).$$

\medskip

We now conclude with the first case in Equation \ref{eq:2} where $K=U, \ell=\ell_1=1$ but for $m\geq 2$. It is easy to see that the exact same argument as above yields the total count of cases as $4(m-1)r(\ell_1)=8(m-1)$ since $r(\ell_1)=2$ for $K=U$. However, a few cases are double counted, and $\Int$ moves are introduced in certain others. We pay attention to these at present.

In the formulation of the previous case, let $m_1$ be different from $0$ and $m-1$. The word $e^{m_1}wf^{m_2}=e^{m_1-1}(ewf)f^{m_2-1}$ has an $\Int$ move when $U$ is written as $w=-$, and similarly for the word $\bar{e}^{m_1}w\bar{f}^{m_2}$ when $U$ is written as $w=+$. Therefore, we have to subtract two for each of the $m-2$ cases. 

The last piece is the double counting.  When $U$ is written in reduced length as $w=-$, we see that $ew=wf$, and when $U$ is written as $w=+$, observe that $\bar{e}w=w\bar{f}$.  This demands that we subtract two for each of the $m-1$ cases when such a pair of words looks the same, resulting in $2(m-1)$ doubly counted scenarios.

Summing up, we get that
$$x^{(3m+1)}_{3m+1}=8(m-1)-2(m-2)-2(m-1)=4(m-1)+2.$$

\bigskip

{\it Part (3)}:  
We begin by writing Definition \ref{def:xy} of $y_i^{(n)}$ with all the terms on one side and using Lemma \ref{lem:binom}.

\begin{align*}
0 &= x_{i+1}^{(n)}-x_i^{(n)}+y_i^{(n)}\\
&= \sum_{j=0}^{n-i-1} \binom{n-i-1}{j}x_{n-3j}^{(n-3j)} - \sum_{j=0}^{n-i} \binom{n-i}{j}x_{n-3j}^{(n-3j)}+y_i^{(n)}\\
&= \sum_{j=0}^{n-i-1} \left[ \binom{n-i-1}{j} - \binom{n-i}{j}\right] x_{n-3j}^{(n-3j)} - x_{n-3(n-i)}^{(n-3(n-i))} +y_i^{(n)}\\
&= \sum_{j=1}^{n-i-1} \left[ \binom{n-i-1}{j} - \binom{n-i}{j}\right] x_{n-3j}^{(n-3j)} - x_{n-3(n-i)}^{(n-3(n-i))} +y_i^{(n)}.\\
x_{3i -2n}^{(3i -2n)} &= \sum_{j=1}^{n-i-1} \left[ - \binom{n-i-1}{j-1} \right] x_{n-3j}^{(n-3j)} +y_i^{(n)}.
\end{align*}

Because of the $2n$ in the index on the left, we need only consider both parity cases for $i$; we choose for convenience the cases $i=3$ and $i=4$ based on Property \ref{property:xy} (Y3) and (Y4), respectively.

For the (easier) case $i=4$ we have $y_4^{(n)}=0$ and obtain:
\begin{equation*}
x_{12 -2n}^{(12 -2n)} = \sum_{j=1}^{n-5} \left[ - \binom{n-5}{j-1} \right] x_{n-3j}^{(n-3j)}.
\end{equation*}

For the remaining case with $i=3$ we have $y_3^{(n)}=x_4^{(n)}$.  With $\binom{k}{-1}$ for $k\geq 0$ set to 0, we use Lemma \ref{lem:binom} to obtain:
\begin{align*}
x_{9 -2n}^{(9 -2n)} &= \sum_{j=1}^{n-4} \left[ - \binom{n-4}{j-1} \right] x_{n-3j}^{(n-3j)} + x_4^{(n)}\\
&= \sum_{j=0}^{n-4} \left[ - \binom{n-4}{j-1} \right] x_{n-3j}^{(n-3j)} + \sum_{j=0}^{n-4} \binom{n-4}{j}x_{n-3j}^{(n-3j)}\\
&= \sum_{j=0}^{n-4} \left[ \binom{n-4}{j} - \binom{n-4}{j-1} \right] x_{n-3j}^{(n-3j)}.
\end{align*}
\end{proof}

\section{Corollaries, Numerical Examples, a Conjecture, and an Open Question}
\label{sec:data}

Our first corollary gives a stronger result than Theorem \ref{thm:noMoves}.
\begin{corollary}
\label{cor:noint}
For any knot $K$, the probability $P(T^{(n)}_K,\not\exists\Int^{(n)})\rightarrow 0$ as $n\rightarrow\infty$.
\end{corollary} 
\begin{proof}
Consider the expressions for the $x^{(n)}_{n}$ with $n=3m+\ell$ given in part two of Theorem \ref{thm:main}. From Property \ref{property:xy} (Sn), we have $x^{(n)}_n=|S^{(n-3)}_{n-2}|$ for $n\geq 4$, but we know that this is the cardinality of the set $\{T^{(n)}_K,\not\exists\Int^{(n)}\}$ from Definition \ref{def:xy}. Since the result in Theorem \ref{thm:main} (2) tells us that these only increase linearly with respect to $n$, the result of the corollary is straightforward since the denominator in the probability expression increases exponentially in $n$.
\end{proof}

We present the following more interesting corollary for knots $K\neq U$.  The proof follows from the fact that Main Theorem \ref{thm:main} (1) and (3) are independent of the knot while (2) depends not on the knot itself but only on the reduced length $\ell$ and the number $r(\ell)$.

\begin{corollary}
\label{cor:coincide}
The probabilities $P(T^{(n)}_{K_1})$ and $P(T^{(n)}_{K_2})$ for two knots $K_1$ and $K_2$ are identical if their reduced lengths and corresponding $r(\ell)$ are the same.
\end{corollary}

Following Corollary \ref{cor:coincide}, we observe  that the probabilities are identical in each of these classes:
\begin{itemize}
	\item the trefoil and figure eight, having reduced length $\ell_1=4$, for any length $n=3i+1$,
	\item the figure eight, $5_1$ and $6_3$, having reduced length $\ell_0=6$, for any length $n=3i$,
	\item $5_2$, $6_1$, and $6_2$, having reduced length $\ell_1=7$, for any length $n=3i+1$, and
	\item $6_1$ and $6_2$, having reduced lengths $\ell_1=7$ and $\ell_0=9$, for any length.
\end{itemize}

We can compute $x^{(i)}_i$ with $i\leq 0$ recursively from Main Theorem \ref{thm:main} (3) despite not having a closed formula.

\begin{figure}[ht]
\begin{center}$
\begin{array}{ccc}
{\includegraphics[width=2in]{unknot.pdf}} &
{\includegraphics[width=2in]{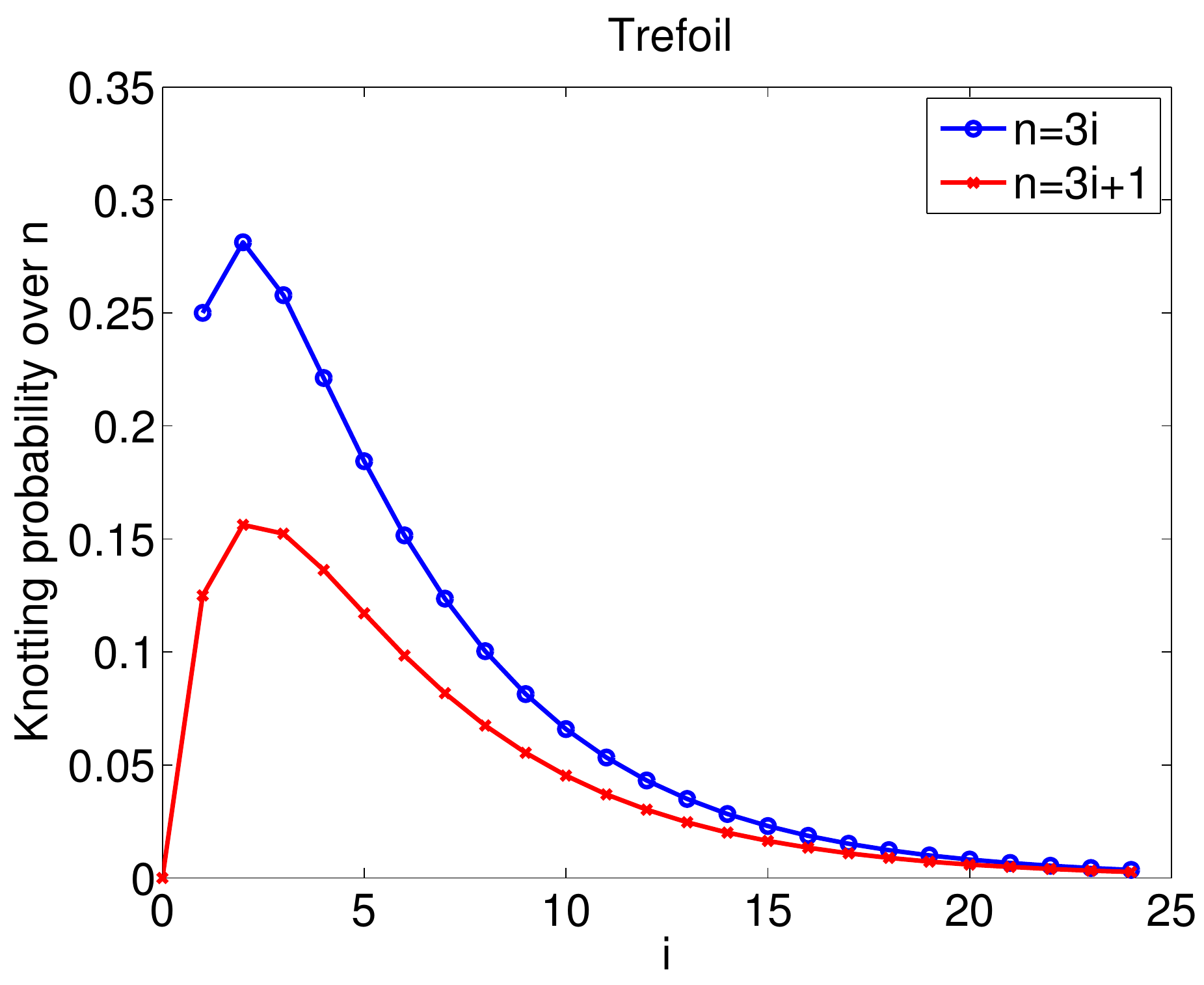}} &
{\includegraphics[width=2in]{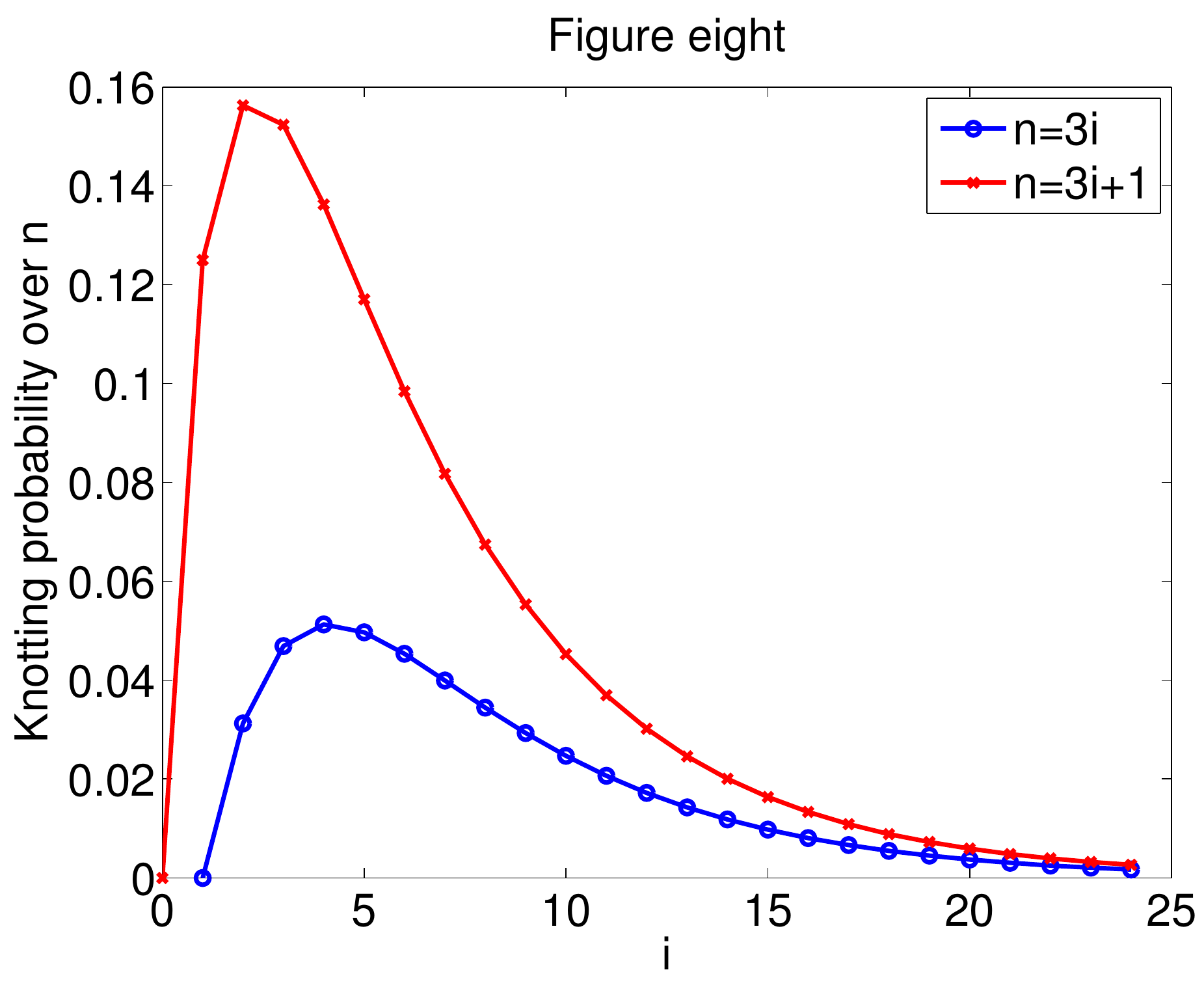}} 
\end{array}$
\end{center}
\begin{center}$
\begin{array}{cc}
{\includegraphics[width=2in]{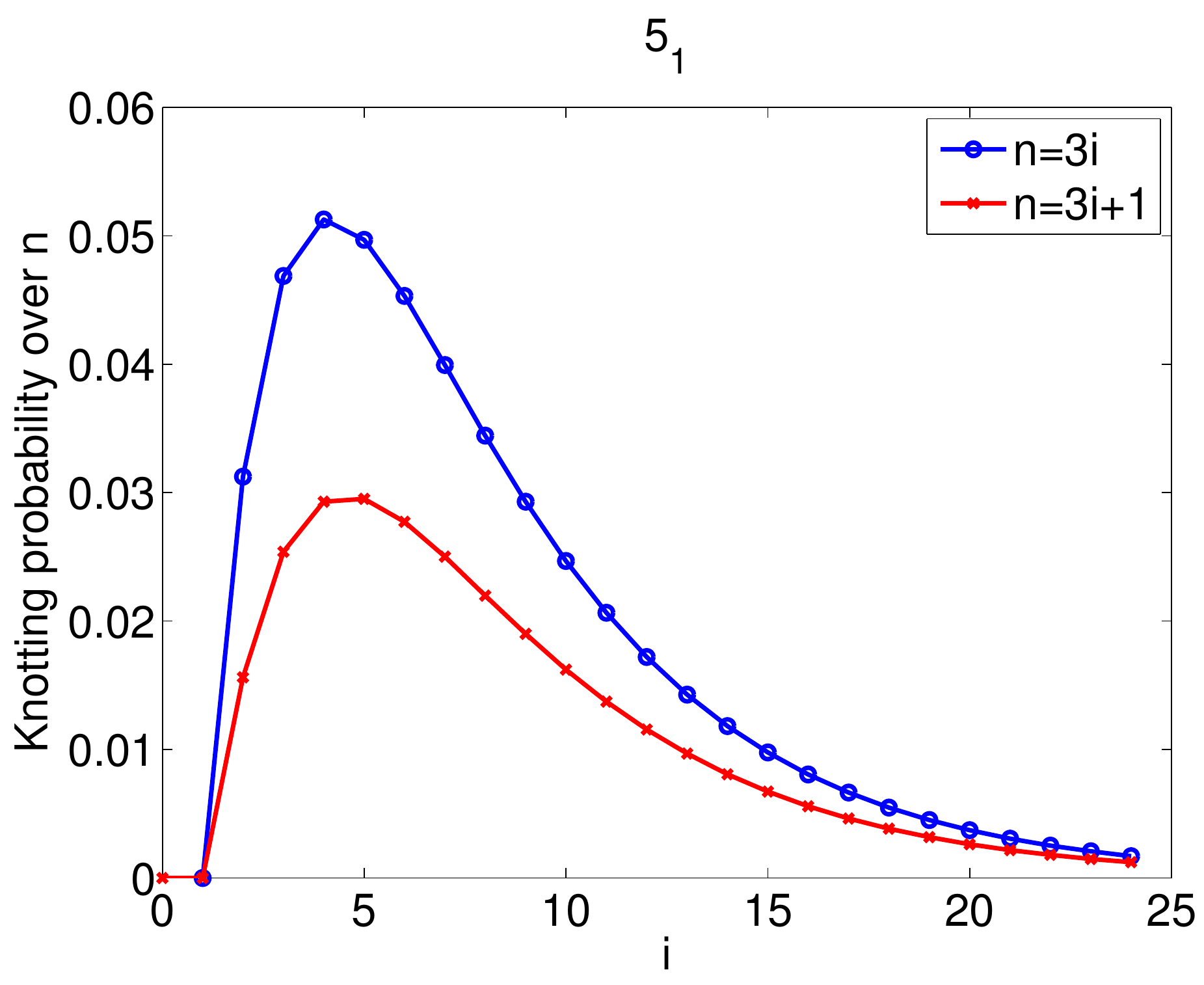}} &
{\includegraphics[width=2in]{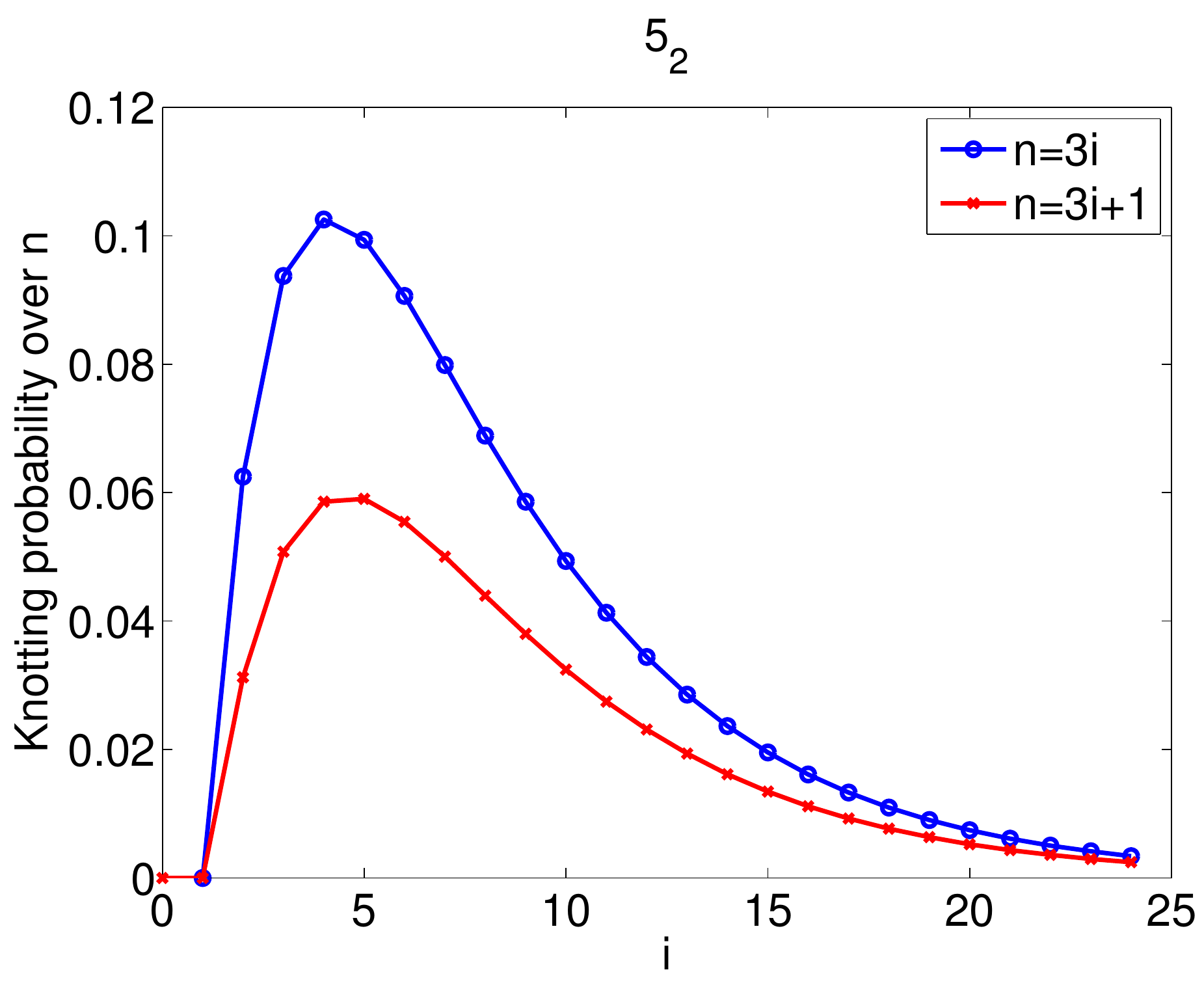}} 
\end{array}$
\end{center}
\begin{center}$
\begin{array}{cc}
{\includegraphics[width=2in]{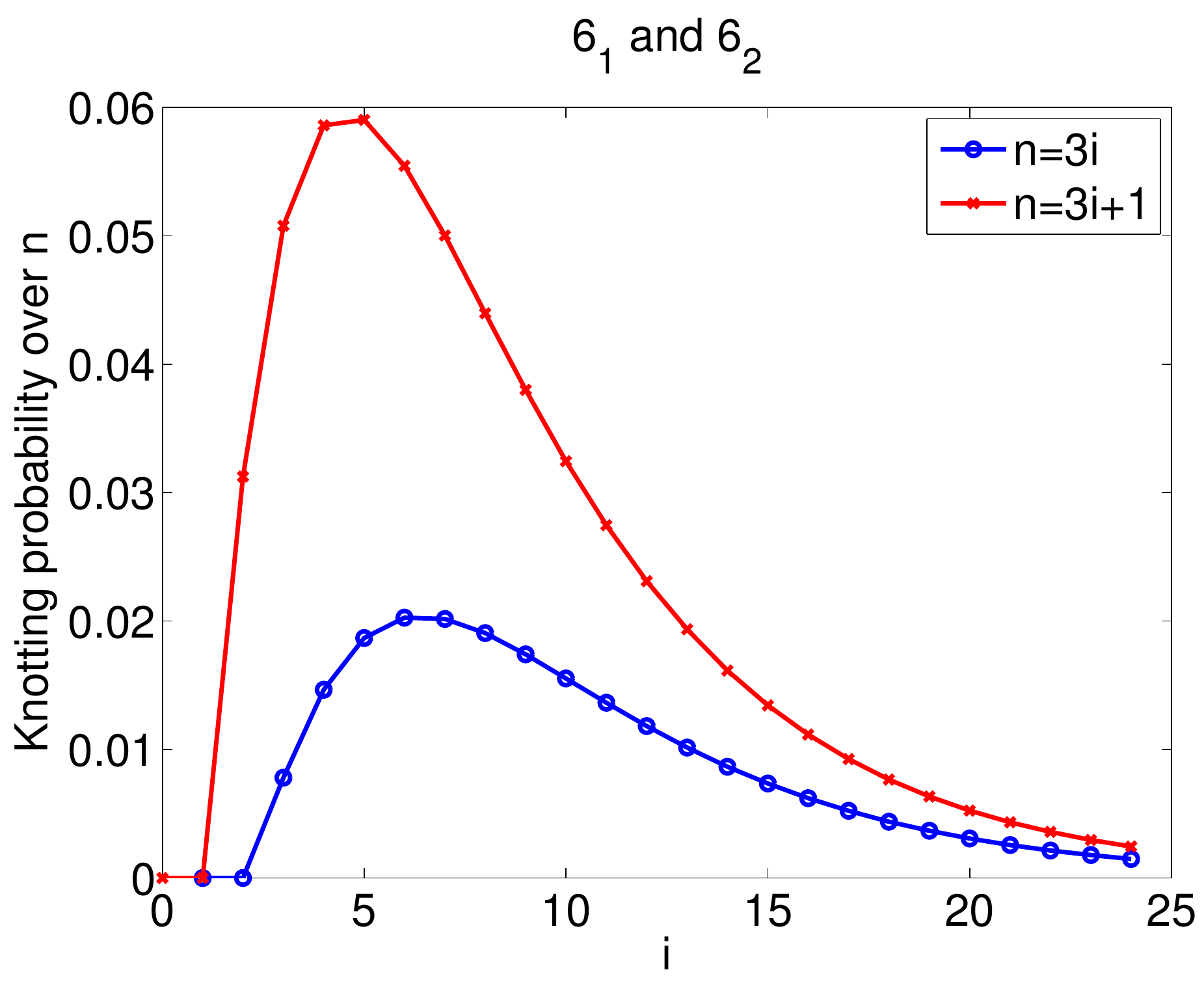}} &
{\includegraphics[width=2in]{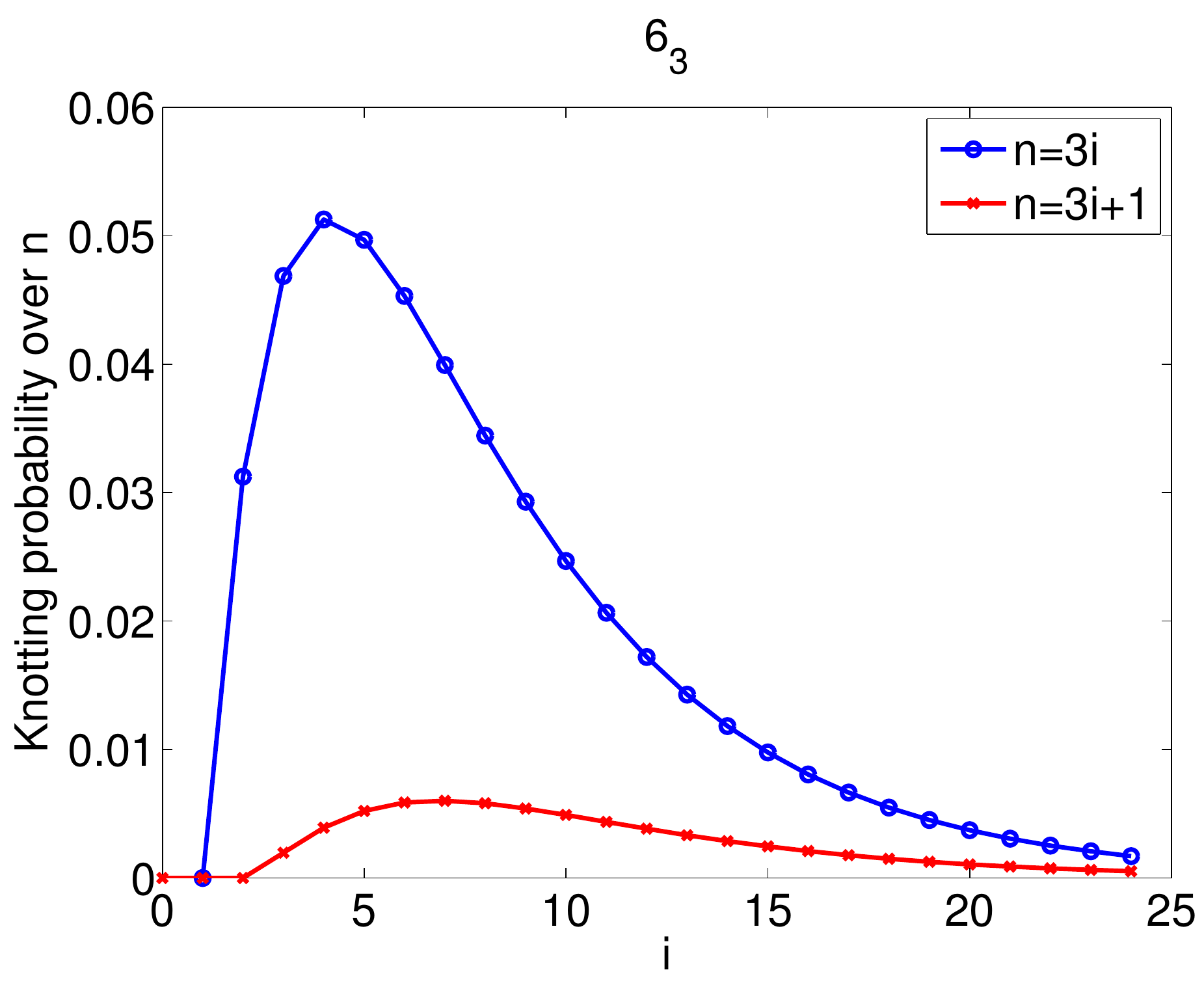}} 
\end{array}$
\end{center}
\caption{Knotting probabilities of the smallest knots (the unknot $U$, the trefoil $3_1$, the figure eight $4_1$, $5_1$, $5_2$, $6_1$, $6_2$, and $6_3$) appearing in $T(3,n+1)$ obtained from MATLAB \cite{MATLAB} using the formulas given in Main Theorem \ref{thm:main}.}
\label{fig:knotprob}
\end{figure}

\begin{computation}
\label{cor:data}
We give in Figure \ref{fig:knotprob} the probabilities for the unknot, trefoil, figure eight, $5_1$, $5_2$, $6_1$, $6_2$ and $6_3$, with $n=3i$ and $3i+1$ for $0\leq i\leq 24$, computed using Equation \ref{eq:maineq} from Main Theorem \ref{thm:main} via MATLAB \cite{MATLAB}. 
Note that $i$ cannot be zero when $n=3i$.
\end{computation}

Notice that the results depicted in Figure \ref{fig:knotprob} follow Corollary \ref{cor:coincide}.

We also observe that the probabilities are identical for $n=3i$ and $3i+1$ cases for the unknot for those $i$ shown in Figure \ref{fig:knotprob}.

We conclude with a conjecture and an open question based on Computation \ref{cor:data} despite not having a closed formula for $P(T^{(n)}_K)$ from Main Theorem \ref{thm:main}.

\begin{conjecture}
\label{conj:0}
The probability $P(T^{(n)}_K)$ of obtaining a knot $K$ in the model $T(3,n+1)$ goes to 0 as $n$ approaches infinity.
\end{conjecture}

\begin{question}
\label{ques:0rate}
At what rate does this probability $P(T^{(n)}_K)$ go to 0?
\end{question}

\newcommand{\etalchar}[1]{$^{#1}$}
\def\cprime{$'$}
\providecommand{\bysame}{\leavevmode\hbox to3em{\hrulefill}\thinspace}
\providecommand{\MR}{\relax\ifhmode\unskip\space\fi MR }
\providecommand{\MRhref}[2]{%
  \href{http://www.ams.org/mathscinet-getitem?mr=#1}{#2}
}
\providecommand{\href}[2]{#2}

\end{document}